\documentclass[12pt]{amsart}
\usepackage{geometry} 
\usepackage{amsthm}
\newtheorem{theorem}{Theorem}
\newtheorem{lemma}{Lemma}
\theoremstyle{definition}

\theoremstyle{remark}

\newtheorem{obs}{Observation}

\newcommand{\N}{\ensuremath{\mathbb{N}}}
\newcommand{\R}{\ensuremath{\mathbb{R}}}

\newcommand{\Lon}{\ensuremath{\mathcal{L}}}

\newcommand{\e}{\varepsilon}

\newcommand{\ff}{\left|}
\newcommand{\rr}{\right|}

\title{Nonlinear second order oscillators off resonance at certain functional spaces}
\author{Adolfo Arroyo Rabasa} 
\date{}

\begin{document}

\maketitle

\begin{abstract} We will deal with the existence of  odd and $T$-periodic solutions of the scalar equation
\begin{equation}u'' + g(u) = k(t),\label{prob}\end{equation}
where $g : \R \to \R$ is and odd function and $k$ is an odd and $T$-periodic function of mean zero. By putting (\ref{prob}) in means of
$$Lu = Nu,$$
where $Lu = u''$ is the linear part and $Nu = k - g(u)$ is the nonlinear part, generally if one denotes by $P_T$ the continuous $T$-periodic functions and by $Q_T$ the continuous $T$-periodic functions of mean zero, then
$$L : P_T\cap{C^2} \to P_T$$
and one have that $Ker(L) = \R$ and $Rank(L) = Q_T$, which is clearly a resonant problem. We will consider the space of odd and $T$-periodic functions where one can avoid resonace. In this space we state two results of existence, one including a priori bounds and one of uniqueness. This results generalize the results obtained by Hamel in \cite{hamel} on the periodic problem for the forced pendulum equation
\end{abstract}

\section{\sc Observations at resonace}

\noindent We denote $\mathcal{H}$ the set of continuous, odd and $T$-periodic functions.
$$\mathcal{H} := \{u \in P_T \, | \, u(t) = -u(-t) \,\, \forall \, t \in [0,T]\}.$$

\begin{obs}$\mathcal{H}$ is a complete normed space.
\end{obs}

\begin{lemma}\label{l2}Let $u \in P_T$ be an odd function, then $\overline{u} := {1\over T}\int_0^T u \, dt = 0$ 
\end{lemma}
\begin{proof}Let
$$\Omega^+ := \{t \in [-T/2,T/2] : u(t) > 0\},$$
$$\Omega^- := \{t \in [-T/2,T/2] : u(t) < 0\}.$$\\
If $t \in \Omega^+ \, \Rightarrow u(t) > 0 \, \Rightarrow -u(-t) > 0 \, \Rightarrow u(-t) < 0 \, \Rightarrow -t \in \Omega^-$,
therefore $-\Omega^+ \subseteq \Omega^-$. In a similar way one proves that $-\Omega^- \subseteq \Omega^+$ and consequently
\begin{equation}\Omega^+ = -\Omega^-.\label{1}\end{equation}\\

From (\ref{1}), it follows that
\begin{equation}\overline{u} = {1 \over T}\bigg(\int_{\Omega^+} u(t) \, \text{d}t + \int_{\Omega^-} u(t) \, \text{d}t\bigg) = {1 \over T}\bigg(\int_{\Omega^-} u(-t) \, \text{d}t + \int_{\Omega^-} u(t) \, \text{d}t\bigg).\label{2}\end{equation}\\

Finally, the odd property of $u$ in (\ref{2}) implies
$$\overline{u} = {1 \over T}\bigg(-\int_{\Omega^-} u(t) \, \text{d}t + \int_{\Omega^-} u(t) \, \text{d}t\bigg) = 0.$$
\end{proof}

\begin{obs}$\mathcal{H} \subset Q_T$.
\end{obs}

\begin{lemma}\label{l1}$L|_\mathcal{H} : dom(L)\cap\mathcal{H} \to \mathcal{H}$ is invertible.
\end{lemma}From the above observation it is clear that $L(\mathcal{H}) \subset P_T$, therefore is necessary to prove that $Lu$ is an odd function for all odd functions $u$.
\begin{proof}For $u'$, we have
\begin{equation}u'(t) = \lim_{h \to 0} {u(t) - u(t+h) \over h} = \lim_{h \to 0} {u(-t) - u(-t+-h) \over -h} = u'(-t),\label{3}\end{equation}
it follows that $u'$ is an even function. Using (\ref{3}) in $u''$ it follows that
$$u''(t) = \lim_{h \to 0} {u'(t) - u'(t+h) \over h} = -\lim_{h \to 0} {u'(-t) - u'(-t+-h) \over -h} = -u''(-t).$$
This proves that $L|_\mathcal{H} : dom(L)\cap\mathcal{H} \to \mathcal{H}$ is well defined. It is well known from \cite{Lazer} the existence of an integral operator $S : Q_T \to P_T$, a right inverse of $L$ such that 
$$\|S(f)\|_\infty \le {T^2 \over 2}\|f\|_\infty$$
and $S(\mathcal{H}) \subseteq dom(L)\cap\mathcal{H}$, for it is sufficient to observe that $Ker(L|_\mathcal{H}) = \{0\}$.
\end{proof}
\noindent Lemma \ref{l1} tells us that (\ref{prob}) is a non-resonant problem in $\mathcal{H}$. Naturally, the question arises to answer when does $Nu \in \mathcal{H}$, in order to express any solution of (\ref{prob}) as a classic fixed point problem of the form
$$u = L^{-1}Nu.$$
From now on, let us put $L := L|_\mathcal{H}$.

\section{\sc Existence of solutions in $\mathcal{H}$ in the sublinear case}

Since odd functions form are closed under composition we get
$$\overline{g(u)} = 0,$$
and even more
$$Nu = k - g(u) \in \mathcal{H}$$
if $g$ is an odd function. The latter discussion lead us to our first result.

\begin{theorem}Let $g : \R \to \R$ an odd sublinear function and $k \in Q_T$ an odd function, then equation 
$$u'' + g(u) = k(t)$$
has an odd and $T$-periodic solution. Even more, the set of solutions is bounded.
\end{theorem}
\begin{proof}Based on the discussion of the first section and the latter remark it is clear that
$$Nu = k - g(u) \in \mathcal{H},$$
and that $u$ is a solution of (\ref{prob}) if and only if $u$ is a fixed point of
$$K(u) = L^{-1}Nu.$$
From Sch\"afer's theorem it suffices to show that the set
$$\Sigma := \{u \in \mathcal{H} : u = \lambda L^{-1}Nu, \lambda \in (0,1]\}$$
is bounded. Let us suppose the opposite and let us take $(u_n) \subset \Sigma$ and $(\lambda_n) \subset (0,1]$ such that
$$u_n = \lambda_nL^{-1}Nu_n$$ and
$$\|u_n\|_\infty \to \infty \quad \text{ as } n \to \infty.$$
The sublinearity of $g$ implies that, for all $\e > 0$ there exists $M := M(\e)$ such that $g(t) < M + \e t.$ For $n \in \N$ it follows that\\
$$
\begin{aligned}
\|u_n\|_\infty & \le \lambda_n\|L^{-1}\|_\Lon\|k - g(u_n)\|_\infty \\
& \le  \lambda_n\|L^{-1}\|_\Lon(\|k\|_\infty + M(\e) + \e\|u_n\|_\infty) \\
& \le  \lambda_n\e\|L^{-1}\|_\Lon\|u_n\|_\infty + C(\e) \qquad \qquad \qquad \qquad (C(\e) > 0). \\
\end{aligned}
$$\\
Consequently,
$$\infty = \lim_{n \to \infty} \|u_n|\|_\infty \le \lim_{n \to \infty}{C(\e) \over 1 - \lambda_n\e\|L^{-1}\|_\Lon},$$
which in any way means that
$$\lim_{n \to \infty} \lambda_n = {1 \over \e\|L^{-1}\|_\Lon}.$$
By putting $\e < \|L^{-1}\|^{-1}$ we get 
$$\lim_{n \to \infty} \lambda_n > 1$$
which contradicts the fact that $(\lambda_n) \subset (0,1]$.
\end{proof}

\section{\sc Uniqueness of solutions in $\mathcal{H}$ under convexity conditions}

\noindent Before stating the main result let us remember that we first considered $L$ as an operator with domain in $P_T$ and therefore $\|S\|_\Lon = \|L^{-1}\|_\Lon$ depends only on $T$.

\begin{theorem}If $g \in C^1(\R)$ is an odd function such that $\|g'\|_\infty < 2/T^2$ and $k \in Q_T$ is and odd function then there exists a unique solution in $\mathcal{H}$ of equation
$$u'' + g(u) = k(t).$$
Which between lines tells us that the uniqueness depends on the period $T$.
\end{theorem} 
\begin{proof}
As in Theorem 1, we look for a fixed point of the equation
$$K(u) = L^{-1}Nu \qquad (u \in \mathcal{H}).$$
The fact that $\|g'\|_\infty < 2/T^2$, implies the existence of $\lambda \in (0,1)$ such that
$$\|L^{-1}\|_\Lon\ff g(x) - g(y) \rr < \lambda\ff x - y \rr \qquad \text{ for all } \, x,y \in \R.$$
Then, if $u,v \in \mathcal{H}$ we obtain
$$
\begin{aligned}
\|K(u) - K(v)\|_\infty & \le \|L^{-1}\|_\Lon\| Nu - Nv \|_\infty \\
& = \|L^{-1}\|_\Lon\| g(v) - g(u) \|_\infty \\
& \le \lambda\|v - u\|_\infty = \lambda\|u - v\|_\infty. 
\end{aligned}
$$
Banach's fixed point theorem guarantees the existence and uniqueness of a solution of (\ref{prob}) in $\mathcal{H}$.
\end{proof}

\nocite{Clapp2}
\bibliographystyle{siam}
\bibliography{mybib}

\end{document}